\documentclass[12pt]{amsart}

\usepackage[margin=1in]{geometry}  
\usepackage{graphics}
\usepackage{amsmath}               
\usepackage{amsfonts}              
\usepackage{amsthm}                
\usepackage{amssymb}
\usepackage{mathtools}   
\usepackage{accents}  
\usepackage{enumerate}
\usepackage[colorlinks=true, citecolor = blue]{hyperref}

\usepackage{multicol}    

\newtheorem{theorem}{Theorem}[section]
\newtheorem{lemma}[theorem]{Lemma}

\newtheorem{definition}[theorem]{Definition}
\newtheorem{remark}[theorem]{Remark}

\def\ssag{Schur $\sigma$-ancestor group\ }
\def\ssags{Schur $\sigma$-ancestor groups\ }
\def\ssagsnsp{Schur $\sigma$-ancestor groups}

\def\spsags{Schur$+1$ $\sigma$-ancestor groups\ }

\def\pssag{pseudo-Schur $\sigma$-ancestor group\ }
\def\pssagnsp{pseudo-Schur $\sigma$-ancestor group}

\def\ssg{Schur $\sigma$-group\ }
\def\ssgnsp{Schur $\sigma$-group}

\def\ssgs{Schur $\sigma$-groups\ }
\def\spsg{Schur$+1$ $\sigma$-group\ }
\def\spsgs{Schur$+1$ $\sigma$-groups\ }

\begin{document}
\title{On certain quotients of $p$-class tower groups of quadratic fields}
\author{Michael R. Bush}
\address{Dept. of Mathematics,
  Washington and Lee University, Lexington, VA 24450, USA.}
\email{bushm@wlu.edu}
\date{August 2023}
\begin{abstract}
A characterization of the quotients of $p$-class tower groups of quadratic fields by terms in the lower $p$-central series plays an important role in the formulation of conjectures by Boston, Hajir and the author about the distribution of such groups as the base field varies.
In this paper, another equivalent criterion is given which 
resolves an issue that arose as part of the group-theoretical calculations carried out in relation to these conjectures.
\end{abstract}

\maketitle

\section{Introduction}
In this paper, we consider certain finite quotients of a type of pro-$p$ group that arises naturally in algebraic number theory. The definitions and main statements in this paper do not require any knowledge of number theory, but we briefly recall some facts and historical developments as motivation.

The class group  $Cl(K)$ of a number field $K$ is an important invariant that indicates whether or not the ring of integers $\mathcal{O}_K$ is a unique factorization domain (UFD). In particular, $Cl(K)$ is a finite abelian group which is trivial exactly when $\mathcal{O}_K$ is a UFD. 
 In the early part of the twentieth century, during the development of class field theory, it was shown that $Cl(K) \cong \mathrm{Gal}(H/K)$ where $H$ is the maximal unramfied abelian extension of $K$, also called the Hilbert class field of $K$. 
During the same period, the study of more general non-abelian extensions with restricted ramification was motivated by various considerations including the question of whether or not a ring $\mathcal{O}_K$ that is not a UFD can always be embedded in the ring of integers $\mathcal{O}_F$ of a larger number field $F$ where $\mathcal{O}_F$ is a UFD. In 1964, this problem was finally shown to have a negative answer by Golod and Shafarevich~\cite{GS} when they established a necessary criterion for a pro-$p$ group to be finite in the form of an inequality relating the generator and relation ranks of the group. This criterion could be shown to fail for the Galois groups attached to certain extensions related to this problem, thus demonstrating that the corresponding extensions were infinite and so yielding examples where no embedding was possible.  We refer the reader to~\cite{R} for a more detailed discussion of this problem and its resolution.

Given their importance, it is natural to try to understand the distribution of $Cl(K)$ as $K$ varies. In the 1980s, Cohen and Lenstra formulated conjectural heuristics~\cite{CL}  regarding the distribution of class groups for various families of fields.  Numerous extensions and generalizations to other contexts have since been proposed. 
One type of generalization is based on the observation that class groups are isomorphic to certain abelian Galois groups as discussed in the previous paragraph. In~\cite{BBH1,BBH2}, non-abelian generalizations of the conjectures for imaginary and real quadratic fields have been formulated by dropping the abelian condition and considering the Galois group of  the maximal unramified $p$-extension of the base field.  These extensions are usually non-abelian and often infinite. The Galois group of such an extension equipped with the Krull topology is a pro-$p$ group which must satisfy several conditions discussed in more detail below. This is the backdrop for the results discussed in this paper, but we note that these conjectures have since been extended beyond the pro-$p$ setting and to other families of fields by Liu, Wood and Zureick-Brown~\cite{LWZ}.

We now recall some definitions and fix notation and terminology.
Let $G$ be a pro-$p$ group and let $G^{ab} = G/[G,G]$ denote the {\em abelianization of $G$}. If we have a surjective homomorphism from $F$ to $G$ where $F$ is a finitely generated free pro-$p$ group and the kernel is finitely generated as a normal subgroup of $F$, then we say that $G$ is {\em finitely presented}. If we consider $\mathbb{F}_p$ as a module with trivial action by $G$, then $g = \dim H^1(G,\mathbb{F}_p)$ and $r = \dim H^2(G,\mathbb{F}_p)$ are the (minimal) {\em generator} and {\em relation ranks for $G$} respectively. See~\cite{K,NSW} for definitions and more background regarding pro-$p$ groups and their cohomology.

The following definition first appears in work of Koch and Venkov~\cite{KV}.
\begin{definition}
Let $G$ be a pro-$p$ group with generator rank $g$ and relation rank $r$. We say that $G$ is a \ssg if the following conditions hold:
\begin{itemize}
\item[(i)] $G^{ab}$ is a finite abelian group.
\item[(ii)] $g = r$.
\item[(iii)] There exists an automorphism $\sigma: G \rightarrow G$ of order $2$ such that the induced automorphism 
$\sigma: G^{ab} \rightarrow G^{ab}$ is the inversion map $\overline{x} \mapsto \overline{x}^{-1}$.
\end{itemize}
\end{definition}

Koch and Venkov observed that if $K$ is an imaginary quadratic field and $G = G_K = \mathrm{Gal}(K^{nr,p}/K)$ where $K^{nr,p}$ is the maximal unramified $p$-extension of $K$, then $G$ is a \ssg and they used this extra structure to strengthen the conclusions of the Golod-Shafarevich criterion in this situation. In particular, they showed that if $p$ is an odd prime and $G$ has generator rank $g \geq 3$, then $G$ must be infinite. A key intermediate step in their work is to show that these groups have presentations of a special form. 

First, one can show using~(iii) that there exists a generating set for $G$ whose elements are inverted by~$\sigma$. Equivalently, if $F$ is free on $x_1, \ldots, x_g$ and we equip $F$ with its own $\sigma$-automorphism defined by $\sigma(x_i) = x_i^{-1}$ for all $i$, then there exists a $\sigma$-equivariant homomorphism from $F$ to $G$. The kernel $R$ of this homomorphism is then not only normal, but invariant under $\sigma$. 
Cohomological arguments imply that the generator and relation ranks are equal, thus there exist relators $r_1, \ldots, r_g$ such that $R = \langle r_1, \ldots, r_g \rangle^F$ where the latter denotes the closed normal subgroup of $F$ generated by these relators. Further arguments then show that the relators can actually be chosen from the set
\[ X = X(F) = \{ r \in \Phi(F) \mid \sigma(r) = r^{-1} \}. \]
Here $\Phi(F)$ is the Frattini subgroup of $F$ and we assume that the generator rank of $F$ is equal to the generator rank of $G$ so $R \subseteq \Phi(F)$.

The existence of such presentations, where the relators lie in $X$, is the starting point for the work in~\cite{BBH1}. Roughly speaking, the model developed there links the probability with which a particular \ssg $G$ arises as a Galois group $G_K$ with the probability that $G$ arises when a presentation of this special type is selected at random. More precisely, we develop a model for how often certain finite quotients of $G$ arise in terms picking random presentations for these quotients and then a limiting process is used to define the probability that $G$ arises. Some important details regarding the latter step for infinite groups $G$ did not appear in~\cite{BBH1} and can be found in~\cite[Appendix 1]{BBH2}.

The quotients of interest here are by terms in the following series of characteristic subgroups.
\begin{definition}
Let $G$ be a pro-$p$ group. The  {\em lower $p$-central series of $G$} is the series of closed subgroups of $G$ defined recursively by $P_0(G) = G$ and 
$P_n(G) = P_{n-1}(G)^p [G,P_{n-1}(G)]$ for $n \geq 1$. Here $P_{n-1}(G)^p$ denotes the closed subgroup generated by $p$th powers of elements in $P_{n-1}(G)$ and $[G,P_{n-1}(G)]$ denotes the closed subgroup generated by all commutators of pairs of elements from $G$ and $P_{n-1}(G)$. Note that $P_1(G) = \Phi(G)$. If $P_c(G) \neq 1$ and $P_{c+1}(G) = 1$, then we say $G$ has $p$-class $c$. We let $G_c = G/P_c(G)$ and refer to this as the maximal $p$-class $c$ quotient of $G$.
\end{definition}

\begin{remark}
If $G$ is a finitely generated pro-$p$ group, then an inductive argument shows that each subgroup in the lower $p$-central series is of finite index and finitely generated. It follows that the maximal $p$-class $c$ quotients $G_c$ are finite $p$-groups in this case. Since every finite quotient of $G$ has finite $p$-class, one can then recover $G$ as the inverse limit $\displaystyle G = \varprojlim G_c$ as  $c \rightarrow \infty$. If $G$ happens to be a finite $p$-group, then it will have finite $p$-class itself and $G_c \cong G$ for all  sufficiently large $c$. In particular, the $p$-class of $G_c$ can be strictly smaller than $c$ in this situation.
\end{remark}

If we apply the construction above to the free group $F$, we obtain a series of increasingly large finite quotients $F_c$. Since the subgroups in the lower $p$-central series are characteristic, the automorphism $\sigma:F \rightarrow F$ defined earlier restricts to each of the quotients $F_c$. We will use the same symbol $\sigma$ for these restricted maps and note that $\sigma$ acts by inversion on $F_c^{ab}$. By analogy with $X$, we define 
\[ X_c = X_c(F) = \{ r \in \Phi(F_c) \mid \sigma(r) = r^{-1} \}. \]

Suppose $G$ is a \ssg and we have $G = F/R$ where $R$ is generated as a normal subgroup by $g$ relators selected from $X$. It is not hard to see that $G_c = F_c/\overline{R}$ where $\overline{R}$ is the normal subgroup generated by the images of the relators under the natural epimorphism from  $F$ to $F_c$ and that these images belong to $X_c$. In fact, there is a converse to this statement which is made precise by the following lemma. 

\begin{lemma}\label{lem-bbh1-2-7}
Let $H$ be a finite $p$-group with $d(H) = g$. 
The following statements are equivalent:
\begin{itemize}
\item[(i)] $H \cong G_c$ for some Schur $\sigma$-group $G$ of generator rank $g$.
\item[(ii)] $H \cong F_c/N$ where $N =  \langle r_1, \ldots, r_g \rangle^{F_c}$ for some $r_1, \ldots, r_g \in X_c$.
\end{itemize}
\end{lemma}
See~\cite[Lemma~2.7]{BBH1} for the proof. Note that we have made a minor change to the statement here to emphasize the equivalence.
In the version of this lemma in~\cite{BBH1}, one of the implications is stated with the assumption that $H$ has $p$-class exactly $c$, but this is unnecessary. Only condition~(i) as stated here is being used in the proof that (i) implies (ii).

\section[A third characterization of \textorpdfstring{$G_c$}{G_c}]{A third characterization of $G_c$}
We now give a different characterization of those normal subgroups $N$ of $F_c$ generated by elements chosen from $X_c$.
\begin{lemma}\label{main}
 Let $N$ be a normal subgroup of $F_c$. The following statements are equivalent:
\begin{itemize}
\item[(i)] $N = \langle r_1, \ldots, r_h \rangle^{F_c}$ for some $r_1, \ldots, r_h \in X_c$.
\item[(ii)] $N$ has the following three properties:
\begin{itemize} 
\item[$\bullet$] $\sigma(N) = N$.
\item[$\bullet$]  $\sigma$ induces the inversion map on the quotient $N/N^p[F_c,N]$.
\item[$\bullet$]  $N/N^p[F_c,N]$ has dimension at most $h$ as an $\mathbb{F}_p$-space.
\end{itemize}
\end{itemize}
\end{lemma}
\begin{proof}
We first show (i) implies (ii). Suppose $N = \langle r_1, \ldots, r_h \rangle^{F_c}$ for some $r_1, \ldots, r_h \in X_c$. From the definition of $X_c$, we have $\sigma(r_i) = r_i^{-1}$  for $1 \leq i \leq h$, hence $\sigma(N) = N$. Since the elements $r_i$ generate $N$ as a normal subgroup, their images $\overline{r_i}$ will generate the quotient $N/N^p[F_c,N]$ as an $\mathbb{F}_p$-space since it is central in $F/N^p[F_c,N]$.  It follows that the dimension of $N/N^p[F_c,N]$  is at most $h$ as an $\mathbb{F}_p$-space. We also see that the induced automorphism $\sigma$ inverts $\overline{r_i}$ and hence every element of the quotient since it is abelian. This establishes (ii).

For the converse, assume that $N$ satisfies the conditions given in (ii). In particular, we suppose $N/N^p[F_c,N]$ has dimension at most $h$ as an $\mathbb{F}_p$-space. We now show that taking any lift of a spanning set gives elements $s_1, \ldots, s_h \in N$ which generate $N$ as a normal subgroup of $F_c$. Suppose that this were not the case and that $H = \langle s_1, \ldots, s_h \rangle^{F_c}$ is properly contained in $N$. Choose $M$ maximal subject to the conditions that $M \trianglelefteq F_c$ and $H \subseteq M \subsetneq N$. Since $N/M$ must intersect nontrivially with the center of $F_c/N$, we deduce from the maximality of $M$ that $N/M$ is central itself. Maximality can  then be applied again to see that $[N:M] = p$. It follows that $N^p [F_c,N] \subseteq M$ and the images of $s_1, \ldots, s_n$ in $N/N^p[F_c,N]$ must then generate a subgroup contained in $M/N^p[F_c,N] \subsetneq N/N^p[F_c,N]$ contradicting our assumption that they form a spanning set for the latter.

Now set $r_i = s_i^{-1} \sigma(s_i) \in X_c$. We observe that $r_i \in N$ since $\sigma(N) = N$ which implies $\langle r_1, \ldots, r_h \rangle^{F_c} \subseteq N$. To see that equality holds, let $\overline{r_i}$ and $\overline{s_i}$ denote the images of $r_i$ and $s_i$ in $N/N^p[F_c,N]$. Then, $\overline{r_i} = \overline{s_i}^{-1} \sigma(\overline{s_i}) = \overline{s_i}^{-2}$. The map $x \mapsto x^{-2}$ is an automorphism of $N/N^p[F_c,N]$ since $p$ is odd, so the elements $\overline{r_i}$ must also span this space. Our earlier argument, with $r_i$ in place of $s_i$, now shows that $N = \langle r_1, \ldots, r_h \rangle^{F_c}$ as desired.
\end{proof}

Combining Lemma~\ref{lem-bbh1-2-7} and Lemma~\ref{main} (with $h = g$), we get:
\begin{lemma}\label{lem-main}
Let $H$ be a finite $p$-group with $d(H) = g$. 
The following statements are equivalent:
\begin{itemize}
\item[(i)] $H \cong G_c$ for some Schur $\sigma$-group $G$ of generator rank $g$.
\item[(ii)] $H \cong F_c/N$ where $N = \langle r_1, \ldots, r_g \rangle^{F_c}$ for some $r_1, \ldots, r_g \in X_c$.
\item[(iii)] $H \cong F_c/N$ where $N$ is a normal subgroup such that:
\begin{itemize} 
\item[$\bullet$] $\sigma(N) = N$.
\item[$\bullet$]  $\sigma$ induces the inversion map on the quotient $N/N^p[F_c,N]$.
\item[$\bullet$]  $N/N^p[F_c,N]$ has dimension at most $g$ as an $\mathbb{F}_p$-space.
\end{itemize}
\end{itemize}
\end{lemma}

In~\cite{BBH2}, the situation for real quadratic fields $K$ is considered. It is shown in~\cite[Lemma 2.5]{BBH2} that $G_K = \mathrm{Gal}(K^{nr,p}/K)$ is a \spsg which is defined as follows.
\begin{definition}
Let $G$ be a pro-$p$ group with generator rank $g$ and relation rank $r$. We say that $G$ is a \spsg if the following conditions hold:
\begin{itemize}
\item[(i)] $G^{ab}$ is a finite abelian group.
\item[(ii)] $r = g$ or $g + 1$.
\item[(iii)] There exists an automorphism $\sigma: G \rightarrow G$ of order $2$ such that the induced automorphisms
on $G^{ab}$ and $H^2(G,\mathbb{F}_p)$ act by inversion.
\end{itemize}
\end{definition}
As we did for \ssgs above, we can combine \cite[Lemma 2.8]{BBH2} and Lemma~\ref{main} (with $h = g + 1$) to get
\begin{lemma}\label{lem-main-2}
Let $H$ be a finite $p$-group with $d(H) = g$. 
The following statements are equivalent:
\begin{itemize}
\item[(i)] $H \cong G_c$ for some \spsg $G$ of generator rank $g$.
\item[(ii)] $H \cong F_c/N$ where $N = \langle r_1, \ldots, r_{g+1} \rangle^{F_c}$ for some $r_1, \ldots, r_{g+1} \in X_c$.
\item[(iii)] $H \cong F_c/N$ where $N$ is a normal subgroup such that:
\begin{itemize} 
\item[$\bullet$] $\sigma(N) = N$.
\item[$\bullet$]  $\sigma$ induces the inversion map on the quotient $N/N^p[F_c,N]$.
\item[$\bullet$]  $N/N^p[F_c,N]$ has dimension at most $g + 1$ as an $\mathbb{F}_p$-space.
\end{itemize}
\end{itemize}
\end{lemma}

In~\cite{BBH1,BBH2}, the quotients $G_c$ are referred to as \ssags or \spsags of $G$ depending on whether $G$ is a \ssg or \spsg respectively. For the remainder of this paper, we will simply refer to such quotients as {\em ancestor groups}. 

\section{Computing ancestor groups}

In~\cite{BBH1,BBH2}, various computations to enumerate ancestor groups $G_c$
were carried using the $p$-group generation algorithm~\cite{O}. This algorithm takes a finite $p$-group $P$ of $p$-class~$c$ and returns all of the finite $p$-groups $Q$ with $p$-class $c+1$ and $Q_c \cong P$. Such groups are called the {\em immediate descendants of $P$}. 

Starting from the elementary abelian $p$-group with generator rank~$g$, one can apply the algorithm recursively to compute all of the finite $p$-groups with generator rank~$g$ until reaching some desired bound on the $p$-class. It can be helpful to picture this process in terms of a tree where groups with the same $p$-class are placed on the same level and edges connect each group to its immediate descendants on the next level.

In practice, such computations rapidly run into storage issues if one attempts to enumerate all $g$-generated finite $p$-groups without further constraint. However, the finite quotients $G_c$ of a \ssg or \spsg $G$ inherit certain properties from $G$ which can be used to narrow the search. In particular, two criteria play an important role in the computations in~\cite{BBH1, BBH2}.

First, the automorphism $\sigma$ on $G$ restricts to give an automorphism of $G_c$ of order $2$ which acts by inversion on its abelianization. It follows that any finite $p$-group not possessing such an automorphism, along with all of its descendants, can be eliminated in our search for ancestor groups.

Second, the relation rank $r(G)$ can be used to bound another quantity associated to $G_c$ which we now introduce. Suppose that $G_c = F/R$ for some normal subgroup $R$ of $F$. Observe that $P_c(F) \subseteq R$ since $P_c(G_c) = 1$. The {\em $p$-multiplicator of $G_c$} is the quotient $R/R^*$ where $R^* = R^p[F,R]$. The {\em $p$-covering group} is the quotient $F/R^*$ and the {\em nucleus of $G_c$} is the subgroup $P_c(F/R^*) = P_c(F)R^*/R^*$ which is contained in the $p$-multiplicator since $R$ contains both $P_c(F)$ and $R^*$.
These quantities are introduced in~\cite{O} and shown to be independent of the choice of $R$ up to isomorphism. They play an important role in the $p$-group generation algorithm. We note that~\cite{O} works with abstract presentations rather than pro-$p$ presentations as we do here, but these quantities still coincide. See~\cite[Remark 2.4]{BBH1} for more details on this point.

The $p$-multiplicator and nucleus are elementary abelian $p$-groups which means their generator ranks coincide with their dimensions as $\mathbb{F}_p$-spaces. We define $h(G_c) = \dim_{\mathbb{F}_p} R/R^* -  \dim_{\mathbb{F}_p} P_c(F)R^*/R^* =  \dim_{\mathbb{F}_p} R/P_c(F) R^*$. By~\cite[Proposition 2]{BN}, we have $h(G_c) \leq r(G)$. In particular, if $G$ is a \ssg with $r(G) = g$, then $h(G_c) \leq g$ for all $c \geq 1$. Similarly, if $G$ is a \spsg with $r(G) = g + 1$, then $h(G_c) \leq g+1$ for all $c \geq 1$. It follows that if a finite $p$-group violates one of these inequalities, then it and its descendants can be eliminated from our search for ancestor groups of the associated type.

Although the conditions above eliminate a lot of groups when searching for ancestor groups of fixed generator rank $g$, there are groups which satisfy these conditions which are not such quotients. For example, when $g$ and $c$ are very small, one can find all possible $g$-generated \ssags of $p$-class $c$ by simply enumerating over tuples of relations $r_1, \ldots, r_g \in X_c$ and keeping the quotients $F_c/\langle r_1, \ldots, r_g\rangle^{F_c}$ which have $p$-class equal to $c$. If $p = 3$, $g = 2$ and $c = 2$, then one finds there are exactly $3$ such groups of $3$-class $2$. These are described in more detail in~\cite[Example~2.9]{BBH1}. Note that the subscript notation used to label these groups in~\cite{BBH1} is not consistent with our use of the notation $G_c$ in this paper. 

On the other hand, if one uses the $p$-group generation algorithm, then starting from the elementary abelian group $\mathbb{Z}/3 \times \mathbb{Z}/3$, one finds that there are $7$ immediate descendants of $3$-class $2$. Of these, $5$ satisfy the conditions above. These include the $3$ groups mentioned in the previous paragraph, but also $2$ additional groups (the abelian groups $\mathbb{Z}/3 \times \mathbb{Z}/9$ and
$\mathbb{Z}/9 \times \mathbb{Z}/9$). The direct enumeration implies immediately that the latter groups cannot be \ssags even though they satisfy the conditions. 

In response to these and other examples, the following terminology was introduced at the end of~\cite[Section~2.4]{BBH1}.
\begin{definition}\label{def-pseudo}
A finite $p$-group $P$ of generator rank $g$ and $p$-class $c$ is said to be a {\em \pssagnsp} if the following conditions hold:
\begin{itemize}
\item[(i)] There exists an automorphism $\sigma: P \rightarrow P$ of order $2$ such that the induced automorphism 
$\sigma: P^{ab} \rightarrow P^{ab}$ is the inversion map $\overline{x} \mapsto \overline{x}^{-1}$.
\item[(ii)] $h(P) \leq g$.
\item[(iii)] $P \ncong G_c$ for any \ssg $G$.
\end{itemize}
\end{definition}
Since enumerating over tuples of relations rapidly becomes infeasible as the parameters increase (see Remark~\ref{remark-enum} below), it would be nice to have a different characterization of such groups. We do this now using part~(iii) of Lemma~\ref{lem-main}.
\begin{theorem}\label{thm-main-cor}
Let $P$ be a finite $p$-group of generator rank $g$ and $p$-class $c$. Then $P$ is a \pssag if and only if:
\begin{itemize} 
\item[(a)] There exists a normal subgroup $N$ of $F_c$ such that $P \cong F_c/N$ and:
\begin{itemize}
\item[(i)] $\sigma(N) = N$.
\item[(ii)] $N/N^p[F_c,N]$ has dimension at most $g$ as an $\mathbb{F}_p$-space.
\end{itemize}
\item[(b)] For all $N$ satisfying the conditions in part~(a), the induced action of $\sigma$ on $N/N^p[F_c,N]$ is \underline{not} inversion.
\end{itemize}
\end{theorem}

\begin{proof}
As noted earlier, if $P$ satisfies~(i) in Definition~\ref{def-pseudo}, then we can find a normal subgroup $R$ of $F$ such that $F/R \cong P$ and $\sigma(R) = R$. We then have $P_c(F) \subseteq R$ since $P$ has $p$-class~$c$. If we let $N = R/P_c(F)$, then $F_c/N \cong P$ and $\sigma(N) = N$ for the induced automorphism $\sigma$. Conversely, given such an $N$, one can see that automorphism induced by $\sigma$ on the quotient $P = F_c/N$ has the desired properties.

With $R$ and $N$ as above, we also have $h(P) =   \dim_{\mathbb{F}_p} R/P_c(F) R^* = \dim_{\mathbb{F}_p} N/N^p [F_c,N]$ since
\[  R/P_c(F) R^* \cong \frac{R/P_c(F)}{P_c(F) R^*/P_c(F)} = N/N^p [F_c,N]. \]
The last equality follows since $N^p [F_c,N]$ is the image of $R^* = R^p[F,R]$ under the natural epimorphism from $F$ to $F_c$.
Hence $h(P) \leq g$ if and only if $N/N^p[F_c,N]$ has dimension at most $g$ as an $\mathbb{F}_p$-space. We have thus shown that conditions~(i) and (ii) in Definition~\ref{def-pseudo} are equivalent to conditions~(i) and~(ii) in~(a).

To finish, observe that if $P$ satisfies~(a), then it will satisfy~(b) if and only if $P \ncong G_c$ for any \ssg $G$ or else we would contradict Lemma~\ref{lem-main}. The result follows.
\end{proof}
\begin{remark}
A similar definition and result for \spsgs simply involves replacing the upper bounds on $h(P)$ and the dimension of 
$N/N^p[F_c,N]$ with $g + 1$ instead of $g$.
\end{remark}

We now briefly outline how one can use Theorem~\ref{thm-main-cor} to test whether a $p$-group is an ancestor group during a recursive search for \ssags with the $p$-group generation algorithm. Algorithms for working with groups (including $p$-groups and their  automorphisms) are discussed in more detail in~\cite{HEO}. Many have been implemented in computer algebra systems such as Magma~\cite{magma} and GAP~\cite{gap}. 

Given a finite $p$-group $G$ of generator rank $g$ and $p$-class $c$ that one wants to test, one first checks for the existence of a generator inverting automorphism $\sigma$ on $G$ and also whether $h(G) \leq g$. If either of these conditions fails to hold, then the group and its descendants are not \ssagsnsp. If it passes these tests, then it might potentially be a \pssag as defined above. 

One can enumerate over the normal subgroups $N$ in Theorem~\ref{thm-main-cor} by considering kernels of $\sigma$-equivariant epimorphisms from $F_c$ onto $G$. These homomorphisms naturally correspond to generating tuples for $G$ selected from $X(G)$ since they are induced by $\sigma$-equivariant epimorphisms from the free group $F$ onto $G$ which factor through $F_c$ since $G$ has $p$-class $c$. 
One notes that homomorphisms $F_c \rightarrow G$ have the same kernel if and only if they differ by an automorphism of $G$, so technically it is only necessary to consider generating tuples for $G$ that lie in distinct orbits under the action by the automorphism group. Further, a tuple generates $G$ if and only if its image generates the Frattini quotient $G/\Phi(G)$ which can be tested quickly.

Given each normal subgroup $N$, one finally checks whether the $\sigma$-automorphism for $F_c$ induces an action by inversion on $N/N^*$. This can be done by testing membership $n_i \sigma(n_i) \in N^*$ given a generating set  $\{ n_i \}_{i=1}^{h(G)}$ for $N$. If action by inversion is found, then we know the group $G$ must be a \ssag by Lemma~\ref{lem-main}. If no such action is found for all $N$, then we conclude the group is a \pssag by Lemma~\ref{lem-main} and it and its descendants can be eliminated from the search. 

\begin{remark}\label{remark-enum}
A more direct approach to checking if a given $g$-generated finite $p$-group $G$ is an ancestor group of $p$-class $c$ would be to compute $X_c$ and then enumerate over  relators $r_1, \ldots, r_g \in X_c$, checking to see if $G \cong F_c / \langle r_1, \ldots, r_g \rangle^{F_c}$. To illustrate how quickly this becomes infeasible as $c$ increases, consider Table~\ref{table1} showing the size of $X_c^g$  for small $c$ and $g$ when $p = 3$.  We have not attempted to compute further entries given how quickly they are increasing. \\[2 pt]
\end{remark}

\begin{table}
\caption{Sizes of $F_c$ and $X_c^g$ for various $g$ and $c$ (with $p = 3$).}
\label{table1}
\begin{tabular}{|c|c|c|c|}
\hline
$g$ & $c$ & $|F_c|$ &  $|X_c|^g$ \\[1 pt]
\hline 
$2$ & $2$ & $3^5$ & $3^4$ \\
$2$ & $3$ & $3^{10}$ & $3^{12}$  \\
$2$ & $4$ & $3^{18}$ & $3^{20}$ \\
$2$ & $5$ & $3^{32}$ & $3^{40}$ \\
\hline
$3$ & $2$ & $3^9$ & $3^9$ \\
$3$ & $3$ & $3^{23}$ & $3^{42}$ \\
\hline
\end{tabular}
\end{table}


\begin{thebibliography}{99}
\bibitem{magma}  W. Bosma, J. Cannon, C. Playoust, {\em The Magma algebra system. I. The user language}, J. Symbolic Comput., {\bf 24} (1997), 235--265. 
\bibitem{BBH1} N.~Boston, M.R.~Bush, F.~Hajir, \emph{Heuristics for $p$-class towers of imaginary quadratic fields}, Math. Ann. {\bf 368} (2017), no. 1--2, 633--669.
\bibitem{BBH2} N.~Boston, M.R.~Bush, F.~Hajir, \emph{Heuristics for $p$-class towers of real quadratic fields}, J. Inst. Math. Jussieu  {\bf 20} (2021), no. 4, 1429--1452.. 
\bibitem{BN} N.~Boston and H.~Nover, {\em Computing pro-$p$ Galois groups}, pp.1--10 in: ANTS VII Proceedings Berlin 2006, Lecture Notes in Computer Science {\bf 4076}, Springer-Verlag, Berlin, 2006.
\bibitem{CL} H.~Cohen and H. W.~Lenstra, Jr., {\em Heuristics on class
  groups of number fields}, pp. 33--62 in: Number theory,
  Noordwijkerhout 1983, Lecture Notes in Mathematics {\bf 1068}, Springer-Verlag, Berlin, 1984.
\bibitem{gap} The GAP~Group, {\em GAP -- Groups, Algorithms, and Programming, Version 4.12.2}; 2022,
  \url{https://www.gap-system.org}.
\bibitem{GS} E. S. Golod, I. R. $\Check{\text{S}}$afarevi$\Check{\text{c}}$, 
{\em On Class Field Towers}, Izv. Ak. Nauk. SSSR {\bf 28} (1964), 273--276
(Russian), AMS Translations {\bf 48} (1965), 91--102.
\bibitem{HEO} D.F.~Holt, B. Eick, E.A.~O'Brien, {\em Handbook of Computational Group Theory}, Chapman and Hall/CRC, 2005.
\bibitem{K} H. Koch. {\it Galois Theory of $p$-Extensions},
Springer-Verlag, Berlin, 2002.
\bibitem{KV} H.~Koch and B.B.~Venkov, {\em \"Uber den $p$-Klassenk\"orperturm eines imagin\"ar-quadratischen Zahlk\"orpers}, Soc. Math. France, Ast\'erisque {\bf 24-25} (1975), 57--67.
\bibitem{LWZ} Y. Liu, M.M.~Wood, D. Zureick-Brown, {\em A predicted distribution for Galois groups of maximal unramified extensions}, preprint, 2022, \url{https://doi.org/10.48550/arXiv.1907.05002}.
\bibitem{NSW} J. Neukirch, A. Schmidt, K. Wingberg, {\it
Cohomology of Number Fields}, 2nd edition, Grundlehren der mathematischen
Wissenschaften {\bf 323}, Springer-Verlag, Berlin, 2008.
\bibitem{O} E.A.~O'Brien, {\em The p-group generation algorithm}, J. Symbolic Comput. {\bf 9} (1990), 677--698.
\bibitem{R} P. Roquette, {\em On Class field towers}, in Algebraic Number
Theory, ed. J. Cassels, A. Fr\"ohlich, Academic Press, 1980.
\end{thebibliography}
\end{document}